%% file: main_min_disc_dist.tex
\pdfoutput=1
\newif\ifonecolumn % Boolean to keep track of column mode (can use to format formulas etc)
\documentclass[journal,twoside,web]{ieeecolor}\onecolumnfalse

\usepackage{generic}
\input{macros.tex}

\title{\LARGE \bf
A Minimum Discounted Reward Hamilton-Jacobi Formulation for Computing Reachable Sets 
}
\author{
Anayo K. Akametalu\textsuperscript{1}, \and Shromona Ghosh\textsuperscript{1}, \and Jaime F. Fisac\textsuperscript{1}, and \and Claire J. Tomlin\textsuperscript{1}
\thanks{
\textsuperscript{1}Department of Electrical Engineering and Computer Sciences, 
        University of California, Berkeley. Cory Hall, Berkeley, CA 94720, United States.\newline
        {\tt\small \{kakametalu, shromona.ghosh, jfisac, tomlin\}~@eecs.berkeley.edu }}% 
\thanks{
This work is supported by NSF under the CPS FORCES and VehiCal projects, by the UC-Philippine-California Advanced Research Institute, and the ONR MURI Embedded Humans. It is also supported through the Army Research Laboratory and was accomplished under Cooperative Agreement Number W911NF-17-2-0196; and in part by Toyota under the iCyPhy center. The research of A.K. Akametalu has received funding from the National GEM Consortium Fellowship.} 
}
\begin{document}
\maketitle
\thispagestyle{empty}
\pagestyle{empty}

\begin{abstract}
We propose a novel formulation for approximating reachable sets through a minimum discounted reward optimal control problem. The formulation yields a continuous solution that can be obtained by solving a Hamilton-Jacobi equation. Furthermore, the numerical approximation to this solution can be obtained as the unique fixed-point to a contraction mapping. This allows for more efficient solution methods that could not be applied under traditional formulations for solving reachable sets. In addition, this formulation provides a link between reinforcement learning and learning reachable sets for systems with unknown dynamics, allowing algorithms from the former to be applied to the latter. We use two benchmark examples, double integrator, and pursuit-evasion games, to show the correctness of the formulation as well as its strengths in comparison to previous work.
\end{abstract}

\section{Introduction \label{sec:intro}}
\input{introduction}

\section{Background \label{sec:back}} 
\input{background}

\section{Minimum of Discounted Rewards \label{sec:mdr}}
\input{discounting}

\section{Improving Convergence \label{sec:conv}}
\input{converge}

\section{Learning Reachable Sets \label{sec:learn}}
\input{learning}

\section{Experiments \label{sec:sim}}
\input{simulations.tex}

\section{Conclusions and Future Work \label{sec:end}}
\input{conclusions}

%\printbibliography
\bibliographystyle{plain}
\bibliography{library}
\newpage
\input{biographies.tex}

\end{document}

%% file: macros.tex
\usepackage{amsmath}
\usepackage{amsfonts}
\usepackage[ruled,vlined,titlenotnumbered]{algorithm2e} 
\usepackage{graphicx} 
\usepackage{subcaption}
\usepackage{amssymb}
\usepackage{bm}
\usepackage{color}
\usepackage{resizegather} 
\usepackage{amssymb}

\usepackage{algorithmic}

\usepackage{textcomp}

\usepackage{todonotes}
\usepackage[bookmarks=true]{hyperref}

\usepackage{ifthen}
\newboolean{include-notes}
\setboolean{include-notes}{true}
 \newcommand{\jfnote}[1]{\ifthenelse{\boolean{include-notes}}%
 {\textcolor{blue}{\textbf{[JF: #1]}}}{}}

\usepackage{enumerate}

\usepackage{url}

\newcommand{\B}{\mathcal{B}}

\newcommand{\D}{\mathcal{D}}

\newcommand{\K}{\mathcal{K}}
\renewcommand{\L}{\mathcal{L}}
\newcommand{\M}{\mathcal{M}}

\newcommand{\R}{\mathcal{R}}

\newcommand{\T}{\mathcal{T}}
\newcommand{\U}{\mathcal{U}}
\newcommand{\V}{\mathcal{V}}

\newcommand{\DD}{\mathbb{D}}

\newcommand{\RR}{\mathbb{R}}

\newcommand{\UU}{\mathbb{U}}

\newcommand{\bu}{\bm{u}}
\newcommand{\bdelta}{\bm{d}}

\newcommand{\bbeta}{\bm{\beta}}

\newcommand{\bx}{\xi}

\newtheorem{definition}{Definition}
\newtheorem{theorem}{Theorem}
\newtheorem{lemma}{Lemma}

\newtheorem{proposition}{Proposition}

\newenvironment{for_journal}{\par\color{black}}{\par}

%% file: introduction.tex
% !TEX root = main_min_disc_dist.tex

Computing reachable sets has been a popular approach for verifying and validating the behavior of dynamical systems. In the reachability problem, one specifies a target set in the state space and then aims to find the set of initial states admitting trajectories that can eventually enter the target within some specified time horizon. The target may represent a region we desire to drive the system to or keep the system away from. Due to its generality, Hamilton-Jacobi (HJ) reachability analysis, which casts the problem in the optimal control framework, has seen broad usage in many safety-critical applications including emergency landing of unmanned aerial vehicles (UAVs) \cite{Ding2016}, vehicle platooning \cite{Chen2015a}, safe learning \cite{Akametalu2014, Gillula2012}, collision-avoidance \cite{Hoffmann2008, Mitchell2005}, and many others \cite{Ding2011a, Huang2011}.

The standard HJ formulation for reachability solves a \emph{minimum reward} (MR) optimal control problem, by computing the viscosity solution of a particular time-dependent HJ partial differential equation (PDE) \cite{Mitchell2005}. Like most optimal control problems, the solution is solved approximately on a grid via value iteration, which recursively applies a \emph{backup operator} until convergence. One downside of MR optimal control problems is that the backup operator in this setting is not a contraction mapping, thus a specific initialization of value iteration is required for convergence to the correct solution. A contraction mapping allows for arbitrary initialization, and with a good initialization convergence can be accelerated. This is particularly useful when computing reachable sets for multiple systems with similar dynamics allowing for the solution of one problem to be used as an initialization for the other.

On the other hand, infinite horizon \emph{sum of discounted reward} (SDR) problems yield backup operators that are contraction mappings \cite{Bertsekas1995}. This allows for more efficient solution methods, like policy iteration \cite{Howard1964, Puterman1979} and multigrid approaches \cite{Alla2015, Chow1991}.   
 
Drawing inspiration from SDR problems, we propose a \emph{minimum discounted reward} (MDR) formulation for computing tight over- and under-approximations of infinite horizon reachable sets. This new formulation yields a contraction mapping, making it possible to extend policy iteration and multigrid approaches to this setting. In addition, it also provides a way forward for learning reachable sets when dynamics are unknown or difficult to model, a topic that has garnered some attention recently \cite{Akametalu2015,Djeridane2006}. This draws greatly from Reinforcement Learning (RL), which attempts to solve the infinite horizon SDR problem when the system model is unknown. Many RL algorithms boil down to finding the fixed-point of the backup operator associated with SDR. The techniques developed in RL, like temporal differencing \cite{Sutton1988} and Q-learning \cite{Watkins1992}, can naturally be extended to do the same in the MDR setting, thus facilitating research in learning reachable sets.

The paper is organized as follows. In Section \ref{sec:back} we briefly go over the MR formulation for HJ reachability analysis, and SDR problems. In Section \ref{sec:mdr} we describe the MDR formulation, and prove some relevant results. In Section \ref{sec:conv} we explore policy iteration and multigrid approaches within the context of the MDR. Section \ref{sec:learn} draws inspiration from RL and presents some preliminary ideas on how the formulation can be used for learning reachable sets. Section \ref{sec:sim} contains examples demonstrating the ideas developed throughout the paper, and we conclude the paper in Section \ref{sec:end}.

Here, both the one-player and adversarial two-player settings are considered, and we opt to use the terminology from optimal control for consistency. For example \emph{optimal control} will be used to refer to both settings, and we use terms like \emph{reward} instead of \emph{payoff} throughout.\footnote{Typically the terms \emph{optimal control} and \emph{reward} are reserved for the one player setting and \emph{zero sum differential game} and \emph{payoff} are the analogous terms in the adversarial two-player setting.}

%% file: background.tex
% !TEX root = main_min_disc_dist.tex

We consider a two-player optimal control problem where the first player (control) wants to keep the system away from the target set for a given time horizon and the second player (disturbance) has the opposite objective.\footnote{We can also consider the case where the objectives are switched, and the control wants to drive the system towards the target, by making minor modifications to the material presented here.} Our goal is to characterize the \emph{infinite-horizon backward reachable set}, which is the set of initial states for which the disturbance wins the game. Going forward all mentions of reachable set refer to the infinite horizon backwards reachable set.

\subsection{System Model \label{subsec:dynamics}}

The analysis in this paper considers a fully observable system whose underlying dynamics may be non-deterministic, but bounded. 
We can formalize this as a dynamical system with state $x\in\RR^n$, and two inputs, $u\in\U\subset\RR^{n_u},  d\in\D\subset\RR^{n_d}$
(with $\U$ and $\D$ compact)
which we will refer to as the \emph{controller} and the \emph{disturbance}:
\begin{equation}\label{fxud}
\dot{x} = f(x,u, d).
\end{equation}

The flow field $f: \RR^n \times \U \times \D\rightarrow\RR^n$ is assumed to be Lipschitz continuous and bounded. In the single-player case we drop the disturbance input, and just have ${f(x,u):\RR^n \times \U \rightarrow\RR^n}$.

Letting $\UU $ and $\DD$ denote the collections of measurable%
	\footnote{A function $f:X\to Y$ between two measurable spaces $(X,\Sigma_X)$ and $(Y,\Sigma_Y)$
	is said to be measurable if the preimage of a measurable set in $Y$ is a measurable set in $X$, that is:
	$\forall V\in\Sigma_Y, f^{-1}(V)\in\Sigma_X$, with $\Sigma_X,\Sigma_Y$ $\sigma$-algebras on $X$,$Y$.}
functions $\bm u: [0,\infty)\to \U $ and $\bm d: [0,\infty)\to \D$ respectively,
and allowing the controller and disturbance to choose any such signals,
the evolution of the system
from any initial state $x$
is determined (see for example \cite{Coddington1955}, Ch.~2, Theorems~1.1,~2.1) by the unique continuous trajectory $\bx:[0,\infty)\to\RR^n$ solving
\begin{equation}\label{eq:xdot}
\begin{split}
\dot{\bx}(s) &= f(\bx(s),\bu(s),\bdelta(s)), \text{ a.e. }s\ge 0,\\
\bx(0) &= x.
\end{split}
\end{equation}
Note that this is a solution in Carath\'eodory's \emph{extended sense}, that is, it satisfies the differential equation \emph{almost everywhere} (i.e. except on a subset of Lebesgue measure zero).

Throughout our analysis, we will use the notation $\bx_{x}^{\bu,\bdelta}(\cdot)$ to denote the state trajectory $t\mapsto x$ corresponding to the initial condition $x\in\RR^n$, the control signal $\bu\in\UU$ and the disturbance signal $\bdelta\in\DD$.

\subsection{HJ Reachability: Minimum Distance to Target}
The target set $\T$ can be implicitly characterized as the sub-zero level set of a Lipschitz \emph{surface function} $l:\RR^n\rightarrow\RR$: 
\begin{equation}\label{eq:l}
x\in\T\iff l(x)<0.
\end{equation}
This function always exists, since we can simply choose the \emph{signed distance function} to $\T$,  $s_{\T}(x)$, which is Lipschitz continuous by construction.\footnote{ For any nonempty set $\M\subset\RR^m$, the signed distance function $s_{\M}:\RR^m\to\RR$  is defined as
$\inf_{y\in\M} |z-y|$ for points outside of $\M$ and $ -\inf_{y\in\RR^m\setminus\M} |z-y|$ for points inside $\M$, where $|\cdot|$ denotes a norm on $\RR^m$.} We use a clipped signed distance since the problem is solved over a fixed domain in practice, $l(x) = \min(\max(s_{\T}(x), -L), L)$ with $L>0$, where $L$ is usually taken to be the largest value (in magnitude) on the domain.

To express whether a given trajectory \emph{ever} enters the target set, let the functional $\V:\RR^n\times\UU\times\DD\to\RR$ assign to each initial state $x$ and input signals $\bu(\cdot)$, $\bdelta(\cdot)$ the lowest value of $l(\cdot)$ achieved by trajectory $\bx_{x}^{\bu, \bdelta}(\cdot)$ over all times $t\ge0$: 
\begin{equation}\label{eq:min_dist_functional}
\mathcal{V}\big(x,\bu(\cdot),\bdelta(\cdot)\big) := \inf_{t\ge 0}l\big(\bx_{x}^{\bu,\bdelta}(t)\big).
\end{equation}
This outcome $\V$, also referred to here as the \emph{minimum of rewards}\footnote{In this context reward refers to $l(x)$, but in general it can be any real-valued function.}, will be nonpositive if there exists any $t\in[0,\infty)$ at which the trajectory enters the target set, and will be strictly positive if the system avoids the target for all $t\ge 0$. 

% Denoting $\V^{\bu,\bdelta}(x) = \V\big(x,\bu(\cdot),\bdelta(\cdot)\big)$, the following statement follows from \eqref{eq:l} and \eqref{eq:min_dist_functional} by construction. 

% \begin{proposition}\label{Value}
% The set of points $x$ from which the system trajectory $\bx^{\bu,\bdelta}_{x}(\cdot)$ under given inputs $\bu(\cdot)\in\UU,\bdelta(\cdot)\in\DD$ will enter the target set $\T$ for some time $t\ge0$ is equal to the zero sublevel set of $\mathcal{V}^{\bu,\bdelta}(\cdot)$:
% \[
% \{x\in\RR^n: \exists t\ge0,\; \bx^{\bu,\bdelta}_{x}(t)\in\T\}=\{x\in\RR^n: \mathcal{V}^{\bu,\bdelta}(\cdot)\le0\}.
% \]
% \end{proposition}

The reachable set can then be obtained by solving a game between the controller and disturbance where the disturbance can use \emph{nonanticipative strategies} to respond to the controller's signal. The disturbance's set of nonanticipative strategies is $\B = \{\bbeta:\UU\to\DD\;|\;
\forall t\ge 0,\; \forall \bu(\cdot),\hat{\bu}(\cdot)\in\UU,$
${\big(\bu(\tau) \!=\! \hat{\bu}(\tau)\text{ a.e.} \tau\ge0\big)\Rightarrow
{\big(\bbeta[\bu](\tau) \!=\! \bbeta[\hat{\bu}](\tau)}{\text{ a.e.} \tau\ge0\big)}}\}$. Intuitively, the disturbance is given an instantaneous informational advantage in the game, but its actions at any given time can only depend on what the controller has done up to that time. With this in place we can define the \emph{value function} $V(x)$ and ultimately the reachable set $\R(\T)$. 
\begin{equation} \label{eq:val_function}
V(x)=\inf_{\beta[\bu](\cdot) \in \B} \sup_{\bu \in \UU}\mathcal{V}\big(x,\bu(\cdot),\beta[\bu](\cdot)\big),
\end{equation}
\begin{equation} \label{eq:reach_set}
\R(\T) = \{x \mid V(x) \le 0\}.
\end{equation}  

The optimization in \eqref{eq:val_function} is referred to here as the minimum reward optimal control problem. It has been shown that the value function for problems with outcome given by \eqref{eq:min_dist_functional} can be characterized as the unique viscosity solution to a particular variational inequality \cite{Barron1989,Barron1990} as shown in \cite{Fisac2015}. When the problem is solved over a finite time interval $[0,T]$, the \emph{finite-horizon value function} $V(x,t)$ can be computed by solving the HJ equation: 
\begin{equation} \label{eq:HJ_mr}
\begin{split} 
    & 0 = \min\left\{l(x)-V(x,t), \frac{\partial V}{\partial t}(x,t)+ \max_{u\in\U} \min_{ d\in\D} \!\!\frac{\partial V}{\partial x}(x,t) f(x,u, d)\right\} \;,\\
    &V(x,T) = l(x).
\end{split}
\end{equation}
 As $T \rightarrow \infty$, $V(x,t)$ becomes independent of $t$. We accordingly drop the dependence on $t$ and recover $V(x)$ as defined in \eqref{eq:val_function}.

\subsection{Computing the Value Function}
Several approximation schemes have been proposed for solving \eqref{eq:HJ_mr} and similar HJ equations on a fixed grid $G$ \cite{Bardi1999, Falcone1994, Mitchell2005,Osher2003,Sethian1996}. Here we will use a semi-Lagrangian approximation based on a discrete time dynamic programming (DP) principle:
\begin{subequations}
\begin{align}
&V_{\Delta t}^{k+1}(x) = \min\{l(x),  \underset{u\in\U}{\max}\text{ }\underset{ d\in\D}{\min} V^{k}_{\Delta t}(x+\Delta t f(x,u,d))\},\\
&V_{\Delta t}^{0}(x) = l(x),\\
&V_{\Delta t} = \lim_{k\rightarrow \infty} V^{k}_{\Delta t},
\end{align}
\end{subequations}%
\noindent where $V_{\Delta t}(x)$ converges to $V(x)$ as the discrete time step $\Delta t \rightarrow 0$. With the semi-Lagrangian approximation, the value function is solved on a discrete grid. Representing the approximation in vectorized form, $\vec{V} \in \RR^{n_G}$, the semi-Lagrangian approach yields \ 
\begin{subequations} \label{eq:dp_min_dist}
\begin{align}
&\vec{V}_{i}^{0} = l(x_i),\\
&\vec{V}_{i}^{k+1} = \min\{l(x_i),  \underset{u\in\U}{\max}\text{ }\underset{ d\in\D}{\min} I[\vec{V}^{k}](x_i+\Delta t f(x_i,u,d)) \label{eq:dp_min_dist_b}\},\\
&\vec{V} = \lim_{k\rightarrow \infty} \vec{V}^{k},
\end{align}
\end{subequations}%
\noindent for  $i=1, ..., n_G$, where $\{x_i\}_{i=1}^{n_G}$ are the grid nodes, $n_G$ is the number of grid nodes, $\vec{V}_i^k$ is the approximate value for $V(x_i, k \Delta t)$ and ${I[\vec{A}]:\RR^n \rightarrow \RR}$ represents an interpolation operator defining, for every point $x$, the polynomial reconstruction based on the values $\vec{A}$. Unless stated otherwise $G$ is taken to be a regular equidistant array of points with mesh spacing $\Delta x_j$ along the $j$th axis, $j=1,...,n$.\footnote{In the most general case the control and disturbance sets are also discretized, $\U=\{u_i\}_{i=1}^{n_\U}$ and $\D=\{d_i\}_{i=1}^{n_\D}$, and the minimax game is approximated.} We use a multilinear interpolator for the interpolation scheme, thus the interpolation function $I[\vec{A}](\cdot)$ is given by a convex combination over the elements of $\vec{A}$,
\begin{equation}
I[\vec{A}](x)= \phi(x)\cdot \vec{A},
\end{equation}%
\noindent where ${\phi(\cdot): \RR^n \rightarrow \RR^{n_G}}$, and $\forall x \in \R^n$ the elements of $\phi(x)$ are nonnegative and sum to 1. 

For conciseness we introduce \emph{control policies} ${\pi(\cdot): \RR^n \rightarrow \U}$ and \emph{disturbance policies} ${\rho(\cdot): \RR^n \rightarrow \D}$, which map from state to control and state to disturbance, respectively. Note that the individual minimax games being played at each grid point can be collectively thought of as a game over policies. We also introduce the \emph{backup operator} ${B[\cdot]: \RR^{n_G} \rightarrow \RR^{n_G}}$, which maps vectorized value functions onto themselves. Define the backup operator as
\begin{equation} \label{eq: op_min_dist}
 B[\vec{A}] := \min\{\vec{l},  \underset{\pi}{\max}\text{ }\underset{ \rho}{\min} \Phi_{\pi, \rho} \vec{A}\},
\end{equation}%
\noindent where ${\vec{l} \in \RR^{n_G}}$ with ${\vec{l}_i = l(x_i)}$ for ${i=1, ..., n_G}$, ${\Phi_{\pi, \rho} \in \RR^{n_G \times n_G}}$ is a policy-dependent stochastic matrix\footnote{$\Phi_{\pi, \rho}$ is effectively the probability transition matrix for a Markov Decision Process (MDP) over a finite state space given by the grid nodes ${\{x_i\}_{i=1}^{n_G}}$.}, and row $i$ of $\Phi_{\pi, \rho}$ is $\phi(x_i+\Delta tf(x, \pi(x_i), \rho(x_i)))$. For the one player setting this matrix becomes $\Phi_{\pi}$. We can now express \eqref{eq:dp_min_dist} more concisely as
\begin{subequations}\label{eq:val_iter}
\begin{align}
&\vec{V}^{0} = \vec{l},\\
&\vec{V}_{i}^{k+1} = B[V^k],\\
&\vec{V} = \lim_{k\rightarrow \infty} \vec{V}^{k}.
\end{align}
\end{subequations}

This recursive procedure \eqref{eq:val_iter} is referred to as \emph{value iteration}, and here $\vec{V}$ is the \emph{vectorized value function}, which is used to approximate the value function $V(x)$ as $I[\vec{V}](x)$. 

The value iteration algorithm in \eqref{eq:val_iter} can be used to solve other optimal control problems, albeit with a different backup operator and initialization $\vec{V}^{0}$. As other optimal control problems are presented we redefine the backup operator and specify the initialization required for the value iteration procedure.

\subsection{Contraction Mappings and Sum of Discounted Rewards}

Value iteration converges to a unique solution, independent of the initialization, when the backup operator is a \emph{contraction mapping}.  
\begin{definition} A mapping $M(\cdot): \RR^{n_G} \rightarrow \RR^{n_G}$, is said to be a contraction mapping in the norm $|| \cdot ||$ over the space $\RR^{n_G}$ if there exists a Lipschitz constant $0\leq \kappa < 1$ such that for any $\vec{A}_1, \vec{A}_2 \in \RR^{n_G}$, $||M(\vec{A}_1) - M(\vec{A}_2)|| \leq \kappa ||\vec{A}_1 - \vec{A}_2||$. 
\end{definition}

Any contraction mapping has a unique fixed point. The operator given by \eqref{eq: op_min_dist} is not a contraction mapping. To see this note that any vector $\alpha \vec{1} \in \RR^{n_G}$ is a fixed point for $\alpha < -L$.

The SDR problem does yield a contraction mapping, and we explore it briefly for the one player case. The objective is to maximize the integral (sum) of exponentially discounted rewards. In an abuse of notation, the associated value function $V(x)$ for this problem is given by
\begin{equation}\label{eq:V_sum_rewards}
V(x) := \sup_{\bu \in \UU} \int_{0}^{\infty}r\big(\bx_{x}^{\bu}(t)\big)\exp(-\lambda t) dt,  \quad \lambda >0,
\end{equation}%
\noindent where $r(\cdot):\RR^n \rightarrow \RR$ is a state-dependent reward function, and $\lambda$ is a \emph{discount rate}.

It is known that this value function is the solution to the time-independent Hamilton-Jacobi equation \cite{Bardi2008},
\begin{equation} \label{eq:HJ_sdr}
\lambda V(x) = \max_{u\in\U} \frac{\partial V}{\partial x}(x) f(x, u)+ g(x).
\end{equation}

Using the same semi-Lagrangian scheme as in the MR setting, \eqref{eq:HJ_sdr} can be approximated as
\begin{equation}
V_{\Delta t}(x) = \max_{u\in\U} \gamma V_{\Delta t}(x + \Delta t\cdot f(x, u))  + \Delta t \cdot r(x),
\end{equation}
 
\noindent where $\gamma=\exp(-\lambda \Delta t)$ is the \emph{discount factor}.

The approximation is solved for using value iteration with the following backup operator:
\begin{equation} \label{eq:backup_sdr}
B[\vec{A}] := \vec{r} +  \underset{\pi}{\max} \gamma \Phi_{\pi} \vec{A},
\end{equation}%
\noindent where $\vec{r}_i =\Delta t \cdot r(x_i)$. The initialization of the value iteration procedure is arbitrary $\vec{V}^{0} \in \RR^{n_G}$ because the backup  operator in \eqref{eq:backup_sdr} is a contraction mapping in the infinity norm, $||\cdot||_\infty$~\cite{Bertsekas1995}.

Discounting reduces the impact of future rewards on the outcome of the trajectory, and ultimately yields a contraction mapping. In fact, the discount factor $\gamma$ is the Lipschitz constant of the contraction mapping.

% \begin{needs_work}
%  The value function at any iteration can be thought of as an estimate of the impact of all future rewards, and the backup operator can be seen as combining these future rewards with the current reward $\vec{r}$. Essentially things that happen far into the future are ``forgotten", which includes errors in the initial estimate $\vec{V}^0$.  
% \end{needs_work}

%% file: discounting.tex
% !TEX root = main_min_disc_dist.tex
Inspired by the SDR setting, we now present the MDR optimal control problem. Moving forward any mention of value function $V(x)$ refers to the MR setting, and is defined by \eqref{eq:min_dist_functional}, \eqref{eq:val_function}, and \eqref{eq:HJ_mr}.

\subsection{What to discount?}

A natural proposal for the outcome of the MDR problem would be
 
\begin{equation}
\inf_{t\ge 0}l\big(\bx_{x}^{\bu,\bdelta}(t)\big)\exp(-\lambda  t).
\end{equation}

However, there is an issue with defining this outcome. Recall that discounting makes rewards contribute less to the outcome the further they occur in the future. Since we take an infimum, the discounted reward should become more positive the further it occurs in the future making it less likely to be selected by the infimum. This only works if the reward is nonpositive everywhere, which is not the case for $l$, since it is a clipped signed distance. This is easily fixed with the following outcome for the MDR problem
\begin{equation} \label{eq:mdr_functional}
\mathcal{Z}\big(x,\bu(\cdot),\bdelta(\cdot)\big) := L + \inf_{t\ge 0}(l\big(\bx_{x}^{\bu,\bdelta}(t)\big)-L)\exp(-\lambda  t) .
\end{equation}%
\noindent The quantity in the infimum is always nonpositive since $L$ upperbounds $l(x)$ by construction. Note that if ${\lambda=0}$, i.e. no discounting, then we have the minimum reward outcome \eqref{eq:min_dist_functional}. 

We define the MDR value function as
\begin{equation} \label{eq:mdr_val_function}
Z(x):=\inf_{\beta[\bu](\cdot) \in \B} \sup_{\bu \in \UU}\mathcal{Z}\big(x,\bu(\cdot),\beta[\bu](\cdot)\big) .
\end{equation}

%
% \begin{proposition}\label{prop:visc}
% The function $U(x)$ is the unique viscosity solution to the following time-independent HJ equation
% \begin{equation} \label{eq:HJ_mdr}
%     0 = \min \left\{h(x)-U(x), \max_{u\in\U} \min_{ d\in\D} \!\!\frac{\partial U}{\partial x}(x) f(x,u, d) - \lambda U(x)\right\} .
% \end{equation}
% \end{proposition}%
% \noindent Due to space constraints we omit the proof for Proposition \ref{prop:visc}, which follows the same structure as the viscosity proofs presented in \cite{Evans1984}. 

%Since $L$ is an additive constant, our attention can be focused on the payoff from the second term of equation (\ref{eq:V_lambda}), which yields a related value function
%
%\begin{equation}
%U_{\lambda}(x)=\inf_{\beta[\bu](\cdot) \in \B} \sup_{\bu \in \UU}\mathcal{V}_{\lambda}\big(x,\bu(\cdot),\bdelta(\cdot)\big)-L, 
%\end{equation}
%\noindent thus $V_{\lambda}(x) = U_{\lambda}(x)+L$. It can be shown that the above value function is the viscosity solution to the following variational equality
%
\input{viscosity_proof}

\subsection{Computing the Discounted Value Function}
The discrete approximation of \eqref{eq:HJ_mdr} is given by
\begin{equation}\label{eq:U_approx}
    U_{\Delta t} (x) = \min\left\{h(x), \max_{u\in\U} \min_{ d\in\D}  \gamma U_{\Delta t}(x+\Delta tf(x,u,d))\right\} ,
\end{equation}%
\noindent which can be solved on a grid $G$ via value iteration 
\begin{subequations} \label{eq:val_iter_backup}
\begin{align}
&\vec{U}^{0} \in \RR^{n_G} ,\\
&\vec{U}^{k+1} = B[U^k] ,\\
&\vec{U} = \lim_{k\rightarrow \infty} \vec{U}^{k} ,
\end{align}
\end{subequations}%
\noindent with the backup operator defined as
\begin{equation} \label{eq:backup_mdr}
B[\vec{A}] := \min\left\{ \vec{h}, \underset{\pi}{\max}\text{ }\underset{ \rho}{\min} \gamma \Phi_{\pi, \rho} \vec{A} \right \}
,
\end{equation}%
\noindent where $\vec{h}_i = h(x_i)$. The MDR value function $Z(x)$ is then approximated by $I[\vec{U}](x)+L$, where again $I[\vec{U}](\cdot)$ is the interpolation operator. We now prove that \eqref{eq:backup_mdr} is a contraction.

\begin{lemma}\label{lem:maxmin} For any two functions $q, g: A \times B \rightarrow \RR$,
\begin{equation}
|\max_a \min_b q(a,b) -\max_a \min_b g(a,b)| \leq \max_a \max_b |q(a,b) - g(a,b)|
.
\end{equation}
\end{lemma}
\begin{proof}
Define the minimax optimizers for $q$ as the pair $(a_q,b_q)$, and minimax optimizers of $g$ as the pair $(a_g, b_g)$. Without loss of generality we assume that ${q(a_q,b_q) \geq g(a_g,b_g)}$.
We then have the following inequalities:
\begin{equation*}
\begin{split}
&|\max_a \min_b q(a,b) -\max_a \min_b g(a,b)|
\leq |q(a_q,b_q) - \min_b g(a_q,b)|\\
&\leq |q(a_q,b_{gg}) - g(a_q,b_{gg})| \leq \max_a \max_b |q(a,b) - g(a,b)|
,
\end{split}
\end{equation*}
with $b_{gg} :=\displaystyle{\arg\min_b g(a_q,b)}$.
\end{proof}
\begin{theorem} 
The operator given by \eqref{eq:backup_mdr} is a contraction mapping in the infinity norm $|| \cdot ||_{\infty}$ on the space $\RR^{n_G}$.
\end{theorem}
\begin{proof} Defining $B[\cdot]$ as in \eqref{eq:backup_mdr}, take $A_1, A_2 \in \RR^{n_G}$:
\begin{equation*}
\begin{split}
&||B[\vec{A}_1] - B[\vec{A}_2]||_{\infty}=\\
&||\min\left\{ \vec{h}, \underset{\pi}{\max}\text{ }\underset{ \rho}{\min} \gamma \Phi_{\pi, \rho} \vec{A}_1 \right \}  - \min\left\{ \vec{h}, \underset{\pi}{\max}\text{ }\underset{ \rho}{\min} \gamma \Phi_{\pi, \rho} \vec{A}_2 \right \}||_{\infty}
 .
\end{split}
\end{equation*}%
\noindent Leveraging the identity $\min\{a,b\} = \frac{1}{2}((a+b)- |a-b|)$ and using the shorthand $\Pi[\vec{A}]=\underset{\pi}{\max}\text{ }\underset{ \rho}{\min} \gamma \Phi_{\pi, \rho} \vec{A}$ , the above is equal to 
\begin{equation*}
\frac{1}{2} ||(\Pi[\vec{A}_1]  - \Pi[\vec{A}_2] ) -  (|\Pi[\vec{A}_1]-\vec{h}|  - |\Pi[\vec{A}_2]-\vec{h}|)||_{\infty} ,
\end{equation*}%
\noindent which by the triangle inequality, is upper bounded by
\begin{equation*}
\frac{1}{2} ||(\Pi[\vec{A}_1]  - \Pi[\vec{A}_2] )||_{\infty} + \frac{1}{2}  ||(|\Pi[\vec{A}_1]-\vec{h}|  - |\Pi[\vec{A}_2]-\vec{h}|)||_{\infty} .
\end{equation*}%
\noindent Given the inequality $|a-b| > |(|a|-|b|)|$, this has upper bound
\begin{equation*}
\begin{split}
&||(\Pi[\vec{A}_1]  - \Pi[\vec{A}_2] )||_{\infty}=\\ 
&||\underset{\pi}{\max}\text{ }\underset{ \rho}{\min} \gamma \Phi_{\pi, \rho}\vec{A}_1 - \underset{\pi}{\max}\text{ }\underset{ \rho}{\min} \gamma \Phi_{\pi, \rho} \vec{A}_2||_{\infty} .
\end{split}
\end{equation*}%
\noindent Finally from Lemma \ref{lem:maxmin}, the last upper bound is 
\begin{equation*}
\underset{\pi}{\max}\text{ }\underset{ \rho}{\max} ||\gamma \Phi_{\pi, \rho} (\vec{A}_1 - \vec{A}_2)||_{\infty} \leq \gamma||\vec{A}_1 - \vec{A}_2||_{\infty} ,
\end{equation*}%
\noindent where the last inequality comes from the fact that $\Phi_{\pi, \rho}$ is a stochastic matrix for all policies, thus $||\Phi_{\pi, \rho}||_{\infty} = 1$.
\end{proof}

\subsection{Under- and Over-Approximating the Reachable Set}
With MDR formulation there is no particular level curve of the value function that characterizes the reachable set. However, it is possible to find level curves that correspond to over and under approximations of the reachable set. 

We have the inequality $Z(x) \geq V(x)$ because the terms being discounted in the outcome are nonpositive. It immediately follows that
\begin{equation} \label{eq:reach_set}
\{x \mid V_{\lambda}(x) \le 0\} \subseteq \R(\T) ,
\end{equation}  

For an over-approximation we first need to characterize the error between $Z(x)$ and $V(x)$. The difference between the two functions can be bounded. Define $\tau(x)$ as the time when the minimum distance to the target is achieved for a trajectory starting at state $x$ under the optimal control and disturbance signals. Then we have the following bound
\begin{equation}
Z(x) - V(x)  \leq (L - l(\bx_{x}^{\bu,\bdelta}(\tau(x))))( 1 -  \exp(-\lambda \tau(x)))
.
\end{equation}%
\noindent Noting that $V(x)=l(\bx_{x}^{\bu,\bdelta}(\tau(x)))$, we get the resulting inequality
\begin{equation} \label{eq:val_error}
Z(x) -  V(x) \exp(-\lambda \tau(x)) \leq L( 1 -  \exp(-\lambda \tau(x))) ,
\end{equation}%
\noindent Furthermore, outside the reachable set $V(x)>0$ leading to
\begin{equation}
Z(x) -  V(x)  \leq L( 1 -  \exp(-\lambda \tau(x))) \quad \forall x \not\in \R(\T) .
\end{equation}

Assuming an upper bound  ${\bar{\tau} \geq \tau(x)}$, we have the following over-approximation for the reachable set
\begin{equation} \label{eq:reach_set}
\R(\T) \subseteq  \{x \mid Z(x) \le L( 1 -  \exp(-\lambda \bar{\tau})) \} .
\end{equation} 

It is clear from \eqref{eq:val_error} that the tightness of the approximations can be tuned via the discount rate $\lambda$.

%% file: viscosity_proof.tex
% !TEX root = main_min_disc_dist.tex
\begin{for_journal}
For convenience we define two functions ${U(x):=Z(x)-L}$ and ${h(x):=l(x)-L}$. 
We will show that $U(x)$ is the viscosity solution to a particular time-independent HJ equation
\begin{equation}\label{eq:HJ_mdr}
    0 = \min\left\{h(x)-U(x), \max_{u\in\U} \min_{ d\in\D} \!\!\frac{\partial U}{\partial x}(x) f(x,u, d) - \lambda U(x)\right\}.
\end{equation}

We begin by presenting some Lemmas to facilitate the proof.
\begin{lemma}
The function $U(x)$ is well defined. 
\end{lemma} 
\begin{proof}
Define the sequence $\{U(x,k)\}_k$, where
\begin{equation}
U(x,k) = \inf_{\beta[\bu](\cdot) \in \B} \sup_{\bu \in \UU} \inf_{t \in [0,k\Delta t]}(h\big(\bx_{x}^{\bu,\beta[\bu]}(t)\big))\exp(-\lambda  t),
\end{equation}
and $\Delta t>0$. The sequence is nonincreasing, since $U(x,k+1) \leq U(x,k)$, and is lower bounded by $-2L$, so it converges. Clearly in the limit this sequence also equals $U(x)$.
\end{proof}

\begin{lemma} \label{dpp}
\emph{Dynamic programming principle.} For $\delta>0$,
\begin{equation}
\begin{split} 
U(x)= \inf_{\beta[\bu](\cdot) \in \B} \sup_{\bu \in \UU_\delta} 
\big [
\min\{&\inf_{t \in [0, \delta]} h\big(\bx_{x}^{\bu,\beta[\bu]}(t)\big))\exp(-\lambda  t), 
\\&\exp(-\lambda \delta)U(\bx_{x}^{\bu,\beta[\bu]}(\delta))\}
\big ]
\end{split},
\end{equation}
where $\UU_{\delta}$ consists of measurable functions on the interval
$[0,\delta]$.
\end{lemma}

\begin{proof}
Splitting the time interval of the infimum in \eqref{eq:mdr_functional} into $[0,\delta]$ and $t>\delta$, $U(x)$ can be expressed as
\begin{equation}
\begin{split}{}{}{}
U(x) =
\inf_{\beta[\bu](\cdot) \in \B} \sup_{\bu \in \UU} 
\big [\min\{&\inf_{t \in [0, \delta]} h\big(\bx_{x}^{\bu,\beta[\bu]}(t)\big))\exp(-\lambda  t),\\ &\inf_{t > \delta} h\big(\bx_{x}^{\bu,\beta[\bu]}(t)\big))\exp(-\lambda  t)\}
\big ]
\end{split}.
\end{equation}
Due to time-invariance of the dynamics, if we define $s=t-\delta$, $u_\delta(s)=u(t+\delta)$ and $y=\bx_{x}^{\bu,\beta[\bu]}(\delta)$,
\begin{equation}
\begin{split}
U(x) = 
\inf_{\beta[\bu](\cdot) \in \B} \sup_{\bu \in \UU} 
\big [\min\{&\inf_{t \in [0, \delta]} h\big(\bx_{x}^{\bu,\beta[\bu]}(t)\big))\exp(-\lambda  t), \\ &\exp(-\lambda \delta) \inf_{s > 0} h\big(\bx_{y}^{\bu_\delta,\beta[\bu_\delta]}(s)\big))\exp(-\lambda s)\}
\big ]
\end{split}.
\end{equation}
The game over the second interval can be optimized independently of the first interval once $y=\bx_{x}^{\bu,\bdelta}(\delta)$ is specified thus it can be replaced with $\exp(-\lambda \delta)U(\bx_{x}^{\bu,\beta[\bu]}(\delta))$. Furthermore $\UU$ is replaced with $\UU_\delta$ since the game is only played explicitly on $[0,\delta]$.
\end{proof}

Before presenting the proof of the viscosity solution, it will be necessary to introduce the concepts of \emph{viscosity subsolution} and \emph{viscosity supersolution}. To simplify notation, we first define the \emph{Hamiltonian} $H:\RR^n \times \RR^n \rightarrow \RR$,
\begin{equation}\label{eq:hamiltonian}
H(x,p) = \max_{u\in\U} \min_{ d\in\D} \!\! f(x,u,d)\cdot p.
\end{equation}

\begin{definition} \label{def:sub_sol}
A function $\phi$ (in this case $U$) on $\RR^n$  is a viscosity subsolution of \eqref{eq:HJ_mdr}, if for all $\psi \in C^1(\RR^n)$ and $x_0$ such that $\phi(x_0) = \psi(x_0)$ and $x_0$ attains a local maximum on $\phi- \psi$, then
\begin{equation}\label{eq:sub_sol}
    \min\left\{h(x_0)-\psi(x_0), H(x_0,\frac{\partial \psi}{\partial x}) - \lambda \psi(x_0)\right\} \geq 0.
\end{equation}
\end{definition}

\begin{definition} \label{def:sup_sol} 
A function $\phi$ (in this case $U$) on $\RR^n$  is a supersolution of \eqref{eq:HJ_mdr}, if for all $\psi \in C^1(\RR^n)$ and $x_0$ such that $\phi(x_0) = \psi(x_0)$ and $x_0$ attains a local minimum on $\phi- \psi$, then 
\begin{equation}\label{eq:sup_sol}
    \min\left\{h(x_0)-\psi(x_0), H(x_0,\frac{\partial \psi}{\partial x}) - \lambda \psi(x_0)\right\} \leq 0.
\end{equation}
\end{definition}

Now we present the major theoretical result for this section.

\begin{theorem}
The function $U(x)$ is the unique viscosity solution to the time-independent HJ equation given by \eqref{eq:HJ_mdr}.

\end{theorem}

\begin{proof}
The structure of the proof follows the classical approach in \cite{Evans1984}, analogously to \cite{Fisac2015}, and draws from viscosity solution theory. We start by assuming that $U$ is not a viscosity solution and then derive a contradiction to Lemma \ref{dpp}. 

A continuous function is a viscosity solution if it is both a \emph{subsolution} and \emph{supersolution}. Note that $U$ is uniformly continuous due to the continuity assumptions on $f$ and $l$. We first show $U$ is a subsolution of \eqref{eq:HJ_mdr}.

% Subsolution proof
From the local maximum condition in Definition \ref{def:sub_sol} and continuity of system trajectories, there exists a sufficiently small $\delta>0$, such that for $\tau \in [0, \delta]$%
\begin{equation*}
U(\bx_{x_0}^{\bu, \bdelta}(\tau)) \leq \psi(\bx_{x_0}^{\bu, \bdelta}(\tau))
\end{equation*}
for all $u(\cdot) \in \UU$ and $d(\cdot) \in \DD$.

For sake of contradiction, assume \eqref{eq:sub_sol} is false, then one of the following must be true
\begin{subequations}
\begin{align}
&h(x_0) = \psi(x_0) - \epsilon_1 \label{eq:sub_contra_a}\\
&H(x_0,\frac{\partial \psi}{\partial x}) - \lambda \psi(x_0) = -\epsilon_2 \label{eq:sub_contra_b},
\end{align} 
\end{subequations}
for some $\epsilon_1, \epsilon_2 > 0$. If \eqref{eq:sub_contra_a} is true, then 
\begin{equation}
h(\bx_{x_0}^{\bu, \bdelta}(\tau))\exp(-\lambda \tau) \leq \psi(x_0) - \frac{\epsilon_1}{2} = U(x_0) - \frac{\epsilon_1}{2}
\end{equation}
Incorporating this into the dynamic programming principle (Lemma \ref{dpp}), we have 
\begin{equation}
\begin{split}
U(x_0) &\leq \inf_{\beta[\bu](\cdot) \in \B} \sup_{\bu \in \UU_\delta}
\big \{\inf_{t \in [0, \delta]} h\big(\bx_{x}^{\bu,\beta[\bu]}(t)\big))\exp(-\lambda  t) \big \} \\&\leq U(x_0) - \frac{\epsilon_1}{2},
\end{split}
\end{equation}
which is a contradiction since $\epsilon_1>0$. Similarly, if \eqref{eq:sub_contra_b}, then for a small enough $\delta>0$ and some nonanticipative strategy $\beta[\cdot]$, 
\begin{equation}
H(\bx_{x_0}^{\bu,\beta[\bu]}(\tau), \frac{\partial \psi}{\partial x}) - \lambda \psi(\bx_{x_0}^{\bu,\beta[\bu]}(\tau)) \leq \frac{\epsilon_2}{2}
\end{equation}
for all $\tau \in [0,\delta]$ and all inputs $\bu(\cdot) \in \UU$. Due to the $\max{}$ in \eqref{eq:hamiltonian}, for $\tau \in [0,\delta]$
\begin{equation}
\begin{split}
&f(\bx_{x_0}^{\bu,\beta[\bu]}(\tau),\bu(\tau), \beta[\bu](\tau)) \cdot \frac{\partial \psi}{\partial x} - \lambda \psi(\bx_{x_0}^{\bu,\beta[\bu]}(\tau)) \leq \\ &H(\bx_{x_0}^{\bu,\beta[\bu]}(\tau), \frac{\partial \psi}{\partial x}) - \lambda \psi(\bx_{x_0}^{\bu,\beta[\bu]}(\tau)).
\end{split}
\end{equation}
Combining the two previous inequalities and integrating over the interval $[0,\delta]$ we have
\begin{equation}
\exp(-\lambda \delta)\psi(\bx_{x_0}^{\bu,\beta[\bu]}(\delta))-\psi(x_0) \leq \frac{\epsilon_2}{2} \delta
\end{equation}
Recalling that $U(x_0)=\psi(x_0)$
\begin{equation}
\exp(-\lambda \delta)U(\bx_{x_0}^{\bu,\beta[\bu]}(\delta)) \leq \frac{\epsilon_2}{2} \delta + U(x_0)
\end{equation}
Incorporating this into the dynamic programming principle, we have
\begin{equation}
U(x_0) \leq \exp(-\lambda \delta)U(\bx_{x_0}^{\bu,\beta[\bu]}(\delta)) \leq \frac{\epsilon_2}{2} \delta + U(x_0),
\end{equation}
which is a contradiction, thus we conclude that $U$ is a subsolution.

% Supersolution proof
Next we show that $U$ is a supersolution. From the local minimum condition in Definition \ref{def:sup_sol} and continuity of system trajectories, there exists a sufficiently small $\delta>0$, such that for $\tau \in [0, \delta]$
\begin{equation*}
U(\bx_{x_0}^{\bu, \bdelta}(\tau)) \geq \psi(\bx_{x_0}^{\bu, \bdelta}(\tau))
\end{equation*}
for all $u(\cdot) \in \UU$ and $d(\cdot) \in \DD$.

If we suppose \eqref{eq:sup_sol} is false, then both of the following must hold 
\begin{subequations}
\begin{align}
&h(x_0) = \psi(x_0) + \epsilon_1 \label{eq:sup_contra_a}\\
&H(x_0,\frac{\partial \psi}{\partial x}) - \lambda \psi(x_0) = \epsilon_2 \label{eq:sup_contra_b},
\end{align} 
\end{subequations}
for some small $\epsilon_1, \epsilon_2 > 0$. If \eqref{eq:sup_contra_a} is true, then 
\begin{equation}
\begin{split}
h(\bx_{x_0}^{\bu, \bdelta}(\tau))\exp(-\lambda \tau) &\geq \psi(x_0) + \frac{\epsilon_1}{2} \\ &= U(x_0) + \frac{\epsilon_1}{2}
\end{split}
\end{equation}
Similarly, if \eqref{eq:sub_contra_b}, then for small enough $\delta>0$ and  some input $\bu(\cdot) \in \UU$ 
\begin{equation}
H(\bx_{x_0}^{\bu,\beta[\bu]}(\tau), \frac{\partial \psi}{\partial x}) - \lambda \psi(\bx_{x_0}^{\bu,\beta[\bu]}(\tau)) \geq \frac{\epsilon_2}{2}
\end{equation}
for all $\tau \in [0,\delta]$ and all nonanticipative strategies $\beta[\cdot]$. 

Due to the $\min{}$ in \eqref{eq:hamiltonian}, for $\tau \in [0,\delta]$
\begin{equation}
\begin{split}
&f(\bx_{x_0}^{\bu,\beta[\bu]}(\tau),\bu(\tau), \beta[\bu](\tau)) \cdot \frac{\partial \psi}{\partial x} - \lambda \psi(\bx_{x_0}^{\bu,\beta[\bu]}(\tau)) \geq \\ &H(\bx_{x_0}^{\bu,\beta[\bu]}(\tau), \frac{\partial \psi}{\partial x}) - \lambda \psi(\bx_{x_0}^{\bu,\beta[\bu]}(\tau)).
\end{split}
\end{equation}
Combining the two previous inequalities and integrating over the interval $[0,\delta]$ we have
\begin{equation}
\exp(-\lambda \delta)\psi(\bx_{x_0}^{\bu,\beta[\bu]}(\delta))-\psi(x_0) \geq \frac{\epsilon_2}{2} \delta
\end{equation}
Recalling that $U(x_0)=\psi(x_0)$
\begin{equation}
\exp(-\lambda \delta)U(\bx_{x_0}^{\bu,\beta[\bu]}(\delta)) \geq \frac{\epsilon_2}{2} \delta + U(x_0).
\end{equation}
Incorporating this into the dynamic programming principle, we have
%

% HELP: How can we split this equation
\begin{equation} 
\begin{split}
U(x) = 
\inf_{\beta[\bu](\cdot) \in \B} \sup_{\bu \in \UU_\delta} 
\big [\min\{ &\inf_{t \in [0, \delta]} h\big(\bx_{x}^{\bu,\beta[\bu]}(t)\big))\exp(-\lambda  t), \\&\exp(-\lambda \delta)U(\bx_{x}^{\bu,\beta[\bu]}(\delta))\}
\big ]\geq
U(x) + \min\{\frac{\epsilon_1}{2}, \frac{\epsilon_2}{2} \delta \},
\end{split}
\end{equation}
which is a contradiction, thus $U$ is also a supersolution.

Since we have shown that $U$ is both a viscosity subsolution and viscosity supersolution of the HJ equation, this completes the proof that $U$ is a viscosity solution of \eqref{eq:HJ_mdr}. Uniqueness follows from the classical comparison and uniqueness theorems for viscosity solutions (see Theorem 4.2 in \cite{Barron1989}).
\end{proof}
\end{for_journal}

%% file: converge.tex
% !TEX root = main_min_disc_dist.tex
In this section we present methods that may yield much faster convergence than value iteration. These approaches have been extensively applied to the SDR setting, and we now apply them to the MDR problem. 

\subsection{Policy Iteration}

In the one player setting (control only), if the backup operator is a contraction mapping the solution can also be obtained via \emph{policy iteration}. First define the \emph{policy backup operator} $B^{\pi}[\cdot]: \RR^{n_G} \rightarrow \RR^{n_G}$, 
\begin{equation} \label{eq:backup_policy}
B^{\pi}[\vec{A}] = \min\left\{ \vec{h}, \gamma \Phi_{\pi} \vec{A} \right \} .
\end{equation}%
\noindent This operator is a contraction mapping. The policy iteration algorithm generates the sequence $\{\vec{U}^{\pi^k}\}_{k}$ according to
\begin{subequations}\label{eq:pi}
\begin{align}
&\vec{U}^{\pi^k} = B^{\pi^k}[\vec{U}^{\pi^k}] \label{eq:pi_a} ,\\
&\pi^{k+1} = \arg\underset{\pi}{\max}B^{\pi}[\vec{U}^{\pi^k}] ,\\ 
&\vec{U} = \lim_{k\rightarrow \infty} \vec{U}^{\pi^k}.
\end{align}
\end{subequations}

Note that ${\vec{U}^{\pi^k}}$, the fixed point of \eqref{eq:pi_a}, is obtained through value iteration with the policy backup. Finding the fixed point of the policy backup is computationally less intensive than the other backup operators presented thus far, since no optimization is performed over policies. The policy iteration algorithm only optimizes over policies when switching policies.  

In practice policy iteration is typically recommended over value iteration because policies can converge faster than values resulting in faster convergence of the algorithm \cite{Russell2003}. A more detailed analysis of policy iteration (as it pertains to SDR) can be found in \cite{Bokanowski2009,Howard1964, Puterman1979}. 

We now prove that policy iteration converges for the MDR problem.
\begin{proposition}
Assuming a finite control set ${\U=\{u_i\}_{i=1}^{n_\U}}$, the policy iteration algorithm converges to the vectorized value function obtained from \eqref{eq:val_iter_backup} without the disturbance.
\end{proposition}%
\noindent \begin{proof}
It's sufficient to show that sequence ${\{ \vec{U}^{\pi^{k}}\}}_k$ is nondecreasing, i.e. $\vec{U}^{\pi^{k+1}} \geq \vec{U}^{\pi^k}$ $\forall k$. Since the number of policies is finite the nondecreasing criterion implies that the sequence of vectors will converge. Also note that $\underset{\pi}{\max}\text{ }B^{\pi}[\cdot] = B[\cdot]$ as defined in \eqref{eq:val_iter_backup} without the disturbance, so the sequence converges to the vectorized value function.

Consider two sequences ${\vec{X}^{i+1}=\min\big\{\vec{U}^{\pi^k}, \gamma \Phi_{\pi^{k+1}}\vec{X}^i \big\}}$ and ${\vec{Y}^{i+1}=\min\big\{\vec{h}, \gamma \Phi_{\pi^{k+1}} \vec{Y}^i \big\}}$, with ${\vec{X}^0=\vec{Y}^0=\vec{U}^{\pi^k}}$.
Since ${\vec{h}\geq \vec{U}^{\pi^k}}$, by \eqref{eq:pi_a}, we have ${\vec{Y}^i \geq \vec{X}^i,\forall i \geq 0}$.
Next, we note that ${\vec{X}^{1}=\min\big\{\vec{U}^{\pi^k},\gamma \Phi_{\pi^{k+1}}\vec{U}^{\pi^k}\big\}}$ ${=\min\big\{\vec{h}, \gamma \Phi_{\pi^{k}}\vec{U}^{\pi^k},\gamma \Phi_{\pi^{k+1}}\vec{U}^{\pi^k}\big\}}$.
Furthermore, the third term in the $\min\{\,\}$ is greater than the second, so $\vec{X}^{1}=\min\big\{\vec{h}, \gamma \Phi_{\pi^{k}}\vec{U}^{\pi^k} \big\} =  \vec{U}^{\pi^k}$, thus ${\vec{X}^{i} =  \vec{U}^{\pi^k},\forall i \geq 0}$.
Lastly, $\lim_{i\rightarrow \infty}\vec{Y}^i= \vec{U}^{\pi^{k+1}}$, which is the fixed-point of the contraction mapping that generates the sequence.
Bringing everything together we have $\vec{U}^{\pi^{k+1}} = \lim_{i\rightarrow \infty}\vec{Y}^i \geq   \lim_{i\rightarrow \infty}\vec{X}^i = \vec{U}^{\pi^{k}}.$
\end{proof}

\subsection{Multigrid Approach}

The accuracy of the approximation scheme depends on the fineness of the discretization. Finer grids have lower approximation error, but at the cost of increased computational effort. For a desired level of accuracy the number of grid points (and thus computation) necessary grows exponentially with the state space dimension, which is the well-known \emph{curse of dimensionality}. 

One possible way to manage this trade-off between computation and accuracy is a multigrid approach. The idea is to first solve for the approximation on a coarse grid (e.g. grid spacing $2\Delta x_j$) and then use the final solution to initialize either value iteration or policy iteration on a finer grid. The procedure can also be stacked, i.e. given $m$ grids of increasing fineness we can produce approximations of increasing accuracy by using each as an initialization for the subsequent grid. This is only possible because contraction mappings allow great flexibility in the initialization, and yield faster convergence for good initializations. Due to the curse of dimensionality obtaining a good approximation is exponentially cheaper on the coarse grid. 

Multigrid approaches have been applied extensively for SDR problems, where empirical and theoretical improvements have been shown \cite{Alla2015, Chow1991}. We compare multigrid approaches to value iteration in Section \ref{sec:sim}.

%% file: learning.tex
% !TEX root = main_min_disc_dist.tex

When the system model is unknown or complex, the reachable set must be learned from data. There are two approaches for this: model-based RL and model-free RL. In this section we will focus on the model-based approach, and conclude with a brief discussion on the model-free approach and its connection to RL. For ease of presentation we consider the one player case.
 % Learning reachable sets for systems with unknown dynamics has 
 % In this section we briefly explore how the MDR formulation allows one to leverage algorithms from RL for the purpose of learning reachable sets when the system model is unknown or when using function approximators to represent the reachable sets for high dimensional systems. A key insight to be taken from here is that much of RL has been aimed at finding the fixed point solution to the backup operator associated with the infinite-horizon SDR problem thus many of the techniques developed for RL can be applied to solving MDR problems with an appropriate change in the operator.

\subsection{Model-based} \label{sec:model_based}

In the model-based approach the model is assumed to be parameterized by a parameter vector $\mu$. %The data consists of a sequence of state transitions obtained from a real system or simulator $\{(x_i, u_i,x_i^+)\}_i$, where $x^+$ is the next state after applying action $u$ at state $x$ and $\Delta t$ is the time step. 
The data is first used to fit the parameters, and then the value function is computed given the model. As more data is collected the process can be repeated. Here we are intentionally vague about the data and the fitting process, and we focus our attention on how to obtain the value function given the fitted model.

The data collection and fitting produce a sequence of parameters $\{\mu_k\}_k$, which in turn corresponds to a sequence of models $\{f_{\mu_k}\}_k$ and vectorized value functions $\{\vec{V}_{\mu_k}\}_k$ for the MR setting and $\{\vec{U}_{\mu_k}\}_k$ for the MDR setting. With the MR formulation the value iteration algorithm must be initialized with $\vec{l}$ every time a new value function is computed. However, with the MDR formulation $\vec{U}_{\mu_k}$ can be used as the initialization when computing $\vec{U}_{\mu_{k+1}}$. Assuming regularity in the dynamics (with respect to $\mu$), if the parameters only deviate slightly between iterations then $\vec{U}_{\mu_{k}}$should be a good approximation of $\vec{U}_{\mu_{k+1}}$, resulting in faster convergence. If this is not the case then $\vec{l}$ can be used as the default initialization. Furthermore, the following classical result on contraction mappings can provide insight on selecting the initialization:  
\begin{proposition} \label{prop:init_dist}
If $M(\cdot): \RR^{n_G} \rightarrow \RR^{n_G}$ is a contraction mapping in the norm $|| \cdot ||$ over the space $\RR^{n_G}$ with Lipschitz constant $0\leq \kappa < 1$ and fixed-point $\vec{A}^*$, then for any ${\vec{A} \in \RR^{n_G}}$,  ${||\vec{A}^* - \vec{A}|| \leq \frac{1}{1-\kappa}||M(\vec{A}) - \vec{A}||}$. 
\end{proposition}

Given Proposition \ref{prop:init_dist}, when computing $\vec{U}_{\mu_{k+1}}$an upper bound can be computed on its distance to $\vec{U}_{\mu_{k+1}}$ and $\vec{l}$ by applying the contraction mapping to each. The initialization can then be selected to minimize this upper bound. In the worst case only one additional backup operation is performed compared to the default case. 

\subsection{Model-free}

Another approach to handling an unknown model, is to compute the value function directly from the data. This is the approach taken in many RL algorithms that attempt to approximate the value function for SDR problems with unknown models.

Temporal difference (TD) learning is at the heart of many of these methods, which includes TD-lambda\cite{Sutton1988}, Q-learning\cite{Watkins1992}, Deep Q Networks\cite{Mnih2016}, and actor-critic methods\cite{Konda2000}. The key idea is to represent the value function with a parametric function approximator\footnote{A simple function approximator would be interpolation on a grid where the parameters are the grid node values.}, define a loss function on the parameters that is minimized when the value function is the fixed-point of the SDR backup operator, and to update the parameters by performing stochastic gradient descent on the loss function with samples from a real system or simulator.

For ease of presentation consider an autonomous system $f(x)$ (which might be due to a fixed policy), and a sequence of state transitions obtained from the system $\{(x_i,x_i^+)\}$, where $x^+=\bx_x(\Delta t)$ and $\Delta t$ is the time step. If the value function is approximated by $V_\theta(x)$ with parameters $\theta$, then the loss function and update rule for TD are given by
\begin{equation}
\begin{split}
&\L(\theta)=\frac{1}{2}\big(V_{\theta}(x) - (r(x)+\gamma V_{\theta}(x^+))\big)^2 ,\\
&\theta^{i+1} \leftarrow \theta^i + \alpha \nabla_{\theta}V_{\theta^i}(x_i)\big(r(x_i)+\gamma V_{\theta^i}(x^+_i) - V_{\theta^i}(x_i)\big),
\end{split}
\end{equation}%
\noindent where $\alpha \in [0,1]$ is the learning rate and $\theta^i$ is the parameters at iteration $i$. A similar idea can be used in the MDR setting. Taking the approximation $U_\theta(x)$ with parameter $\theta$, the loss function and update rule for the MDR setting would be
\begin{equation}
\begin{split}
&\L(\theta)=\frac{1}{2}\big(U_{\theta}(x) - \min\{h(x),\gamma U_{\theta}(x^+)\}\big)^2,\\
&\theta^{i+1} \leftarrow \theta^i + \alpha\nabla_{\theta}U_{\theta^i}(x_i)\big(\min\{h(x_i),\gamma U_{\theta^i}(x^+_i)\} - U_{\theta^i}(x_i)\big).
\end{split}
\end{equation}

A similar modification can be made to the other TD-based algorithms, in particular those that take into account control actions. We leave the investigation of these ideas for future work.

%% file: simulations.tex
% !TEX root = main_min_disc_dist.tex
This section uses two benchmark models, double integrator and pursuit-evasion game, to exemplify the numerical properties of the MDR formulation. The double integrator model will be used to display the over-/under-approximation of the reachable set, as well as to compare policy iteration with value iteration. Both of the benchmarks will be used to demonstrate the advantages of multigridding, and initializing value iteration with pre-computed solutions to similar problems. 

Unless stated otherwise all algorithms are initialized with $\vec{h}=\vec{l}-L$, and are considered converged when the distance (in the infinity norm) between consecutive iterates falls below $\epsilon =.001$. All experiments were run on a 2016 MacBook Pro with Core i7 processor and 16GB RAM.

\input{double_integrator.tex}

\input{pursuit_evasion.tex}

%% file: double_integrator.tex
% !TEX root = main_min_disc_dist.tex
\subsection{Double Integrator}
The doube integrator consists of two states $(x_1, x_2)$ , and  control $u \in [-u_{max}, u_{max}]$ with dynamics,
\begin{equation}
\begin{split}
\dot{x_1} & = x_2 \\
\dot{x_2} & = u 
\end{split}
\end{equation}
\noindent The state space is discretized into a $161 \times 161$ grid on the domain $[-1,5] \times [-5,5]$, and $u_{max}=2$.

The task is to keep the state trajectory inside the box $\K = [0,4] \times [-3,3]$, thus the target is taken to be its complement $\T=\K^C$. For ease of exposition we define the \emph{safe set} $\Omega(\T) := \R(\T)^C$.

We first show, in Fig.~\ref{fig:convergence} that different level curves of $Z$ over approximates~(bold line) and under approximates~(dotted lines) the analytic safe set~(shown in black) for two values of $\lambda = 0.1, 0.2$, with $\bar{\tau}=2$. As expected smaller values of $\lambda$, yield tighter approximations. 

To verify the convergence properties of the MDR formulation we initialize value iteration with the zero vector $\vec{0}\in \RR^{N_G}$ with $\lambda=0.1$. The error (in the infinity norm) between the converged solution and the one visualized in Fig.~\ref{fig:convergence} is $0.000299$, suggesting convergence to the same fixed point. Under the MR setting value iteration fails to converge with this particular initialization.

\begin{figure}
\includegraphics[scale=0.5]{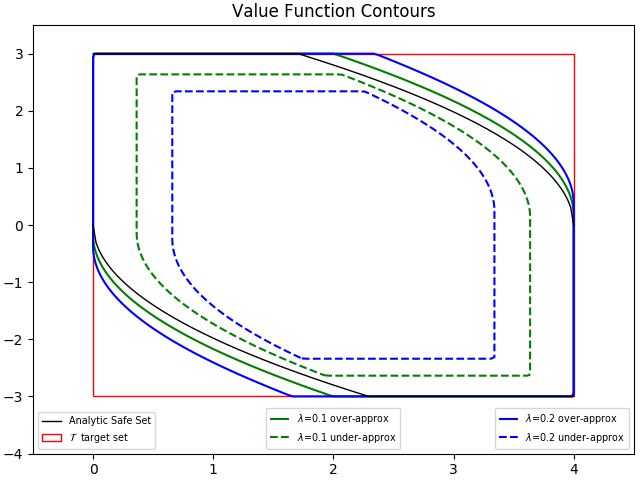}
\caption{The analytic safe set and target set $\mathcal{T}$ are shown in black (interior) and red (exterior), respectively. The over and under approximated $Z$ are shown in bold green and dotted green for $\lambda=0.1$; and bold blue and dotted blue line for $\lambda  = 0.2$ (all interior).}
\label{fig:convergence}
\end{figure}

We now compare value iteration and policy iteration with increasing number of discrete actions in Table~\ref{tab:v_vs_p}. In the table we see that the runtime of value iteration increases linearly with the increase in the number of actions, and policy iteration scales much better. In our particular implementation, the overall runtime favors value iteration, but it is important to note that the majority of the time in policy iteration is spent constructing $\Phi_{\pi_u}$, which is denoted by $T_{\Phi_{\pi_u}}$.\footnote{The data structure used to represent the interpolation is very efficient for sparse matrix multiplication, but is not ideal for indexing, which is necessary to create $\Phi_{\pi_u}$, and results in a relatively large $T_{\Phi_{\pi_u}}$.} Excluding this cost, policy iteration becomes more attractive.

\begin{table}
\centering
\caption{Value Iteration vs Policy Iteration} CPU time in seconds \\
\label{tab:v_vs_p}
\begin{tabular}{|c| c| c| c|}
\hline
\# actions & VI & \multicolumn{2}{|c|}{Policy Iteration} \\ \cline{3-4}
 &  $T_{total} $ & $T_{total}$ & $T_{total} - T_{\Phi_{\pi_u}}$ \\ \hline
2 & 1.468  & 78.197  & 0.102 \\ \hline
50 &  8.753 &  302.456 & 1.007 \\ \hline
250 & 36.973 &  308.565 & 4.318 \\ \hline
500 & 65.305 &  326.280 & 9.760\\
\hline
\end{tabular}
\end{table}

Next, we look at a multigrid approach versus value iteration. For the multigrid approach we also need to run value iteration on a coarse grid, which we construct to have half the resolution per dimension of the nominal grid, e.g. if the nominal grid has $41^2$ nodes then the coarse grid has $21^2$ nodes. The results are shown in Table~\ref{tab:multigrid_di}. We first run value iteration with the standard initialization on both the coarse and fine grid. This produces the values in columns two and three. We then run value iteration on the fine grid initialized with the coarse solution (CS), which makes up column four. Column five (multigrid) is obtained by adding columns two and four. From the table it is clear that the multigrid approach outperforms value iteration, especially as the number of nodes increases.

\begin{table}
\centering
\caption{Double Integrator: Value iteration (VI) with Multigrid} CPU time in seconds
\label{tab:multigrid_di}
\begin{tabular}{|c| c| c| c| c| }
\hline
\# nodes & Coarse grid & Fine grid &  Fine grid-CS & Multigrid \\ \hline
$40^2$ & $0.012$ & $0.025$ & $0.019$ & $0.031$ \\ \hline
$80^2$ & $0.019$ & $0.116$ & $0.041$ & $0.060$\\ \hline
$160^2$ & $0.153$ & $1.110$ & $0.084$ & $0.237$\\ \hline
\end{tabular}
\end{table}

Lastly, denoting the current model as the nominal model $M_n$, we construct two different models: a heavy model $M_h$ with $u_{max}=1.0$, and a light model $M_l$ with $u_{max}=4.0$. This can be interpreted as two systems that have different control authorities due to their different masses. The $M_n$ MDR value function will be less than that of model $M_l$, but greater than that of $M_h$. We compute the value functions for $M_l$ and $M_h$ both initialized with the default initialization, and with the solution for model $M_n$, which we refer to as a \emph{warm start} (WS). This experiment is motivated by Section \ref{sec:model_based}, where we discussed the computation of reachable sets as the system model changes due to new observations from the system. Here we are not concerned with how the model estimates are obtained, but rather how to produce the reachable set for the latest model estimate as quickly as possible. Within this context $M_n$ can be viewed as an old model, and $M_l$ and $M_h$ can be thought of as two possible new models that were inferred from observations. The results for the experiment are shown in Table~\ref{tab:ws_di}. In both cases ($M_l$ and $M_h$) we see that leveraging the solution for $M_n$ improves the convergence time.

\begin{table}
\centering
\caption{Double Integrator: Value Iteration (VI) with Warm Start (WS)} CPU time in seconds
\label{tab:ws_di}
\begin{tabular}{|c| c| c| c| c| c|}
\hline
\# nodes & $M_n$ & $M_l$ &  $M_l$-WS & $M_h$ & $M_h$-WS \\ \hline
$40^2$ & $0.029$ & $0.044$ & $0.041$ & $0.022$ & $0.014$\\ \hline
$80^2$ & $0.205$ & $.364$ & $0.273$ & $0.097$ & $0.095$\\ \hline
$160^2$ & $1.444$ & $2.727$ & $2.304$ & $1.257$ & $1.116$\\ \hline
\end{tabular}
\end{table}

%% file: pursuit_evasion.tex
% !TEX root = main_min_disc_dist.tex
\subsection{Pursuit-Evasion Game}

We now consider the pursuit-evasion game described in \cite{Mitchell2005}. In the game player I (the control) tries to avoid being captured by player II (the disturbance) on a two dimensional plane. Each player is modeled as a simple kinematic point object with planar position and heading, fixed linear velocity and controllable angular velocity. Taking player I to be at the origin the states $(x_1, x_2, x_3)$ are the relative position and heading of player II and the dynamics are

\begin{equation}
\begin{split}
&\dot{x_1}= -v_u+v_d \cos x_3 + ux_2\\ 
&\dot{x_2}= v_d \sin x_3 - ux_1\\ 
&\dot{x_3}= d-u
\end{split}
\end{equation}

The state space is over the domain $[-6,20] \times [-10,10] \times [0,2\pi[$ with $\U=[-u_{max},u_{max}]$ and $\D=[-d_{max}, d_{max}]$.

Player I is considered captured when the relative distance (in position) between both players is less than $R>0$, thus the target set is given by

\begin{equation}
\T= \{x| x_1^{2}+x_2^2 < R^2\}
\end{equation}

We first compute the value functions for the MR and MDR on a $41 \times 41 \times 41$ grid, and setting the model parameters to $v_u=v_d=5$, $u_{max}=d_{max}=1$, and $R=5$. This will be referred to as the nominal model $M_n$. A visualization of the zero sub-level set for both $V(x)$ and $Z(x)$ for $\lambda=0.001$ is shown in Fig.~\ref{fig:air3D}.

% \begin{figure}
% \includegraphics[trim= 2.5cm 0cm 0cm 0.5cm, clip=true,scale=0.65]{air_3D.png}
% \caption{The target set $\mathcal{T}$ is the blue cylinder. The zero sub-level sets of $V$ and $Z$ are in red and green, respectively. The discount rate is $\lambda=0.01$.}
% \label{fig:air3D}
% \end{figure}

\begin{figure*}[h]
    \centering
    \begin{subfigure}[t]{0.3\textwidth}
        \centering
        \includegraphics[trim= 4cm 0cm 0cm 0cm, scale=0.45]{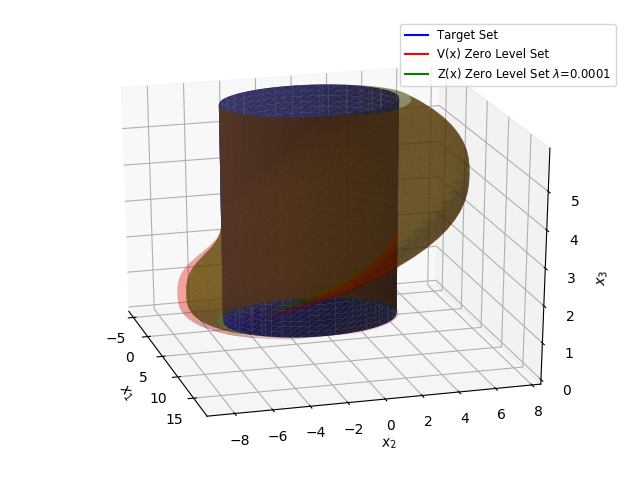}
    \end{subfigure}%
    ~ 
    \begin{subfigure}[t]{0.3\textwidth}
        \centering
        \includegraphics[trim= 3cm 0cm 0cm 0cm, clip=true, scale=0.45]{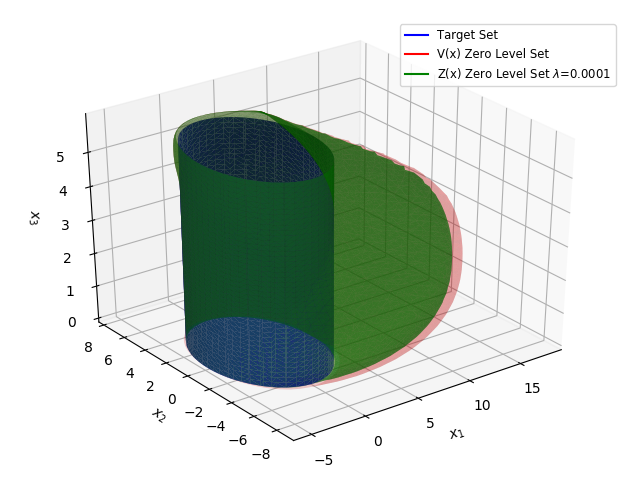}
        %\caption{}
    \end{subfigure}
    ~
    \begin{subfigure}[t]{0.3\textwidth}
        \centering
        \includegraphics[trim= 3cm 0cm 0cm 0cm, clip=true, scale=0.45]{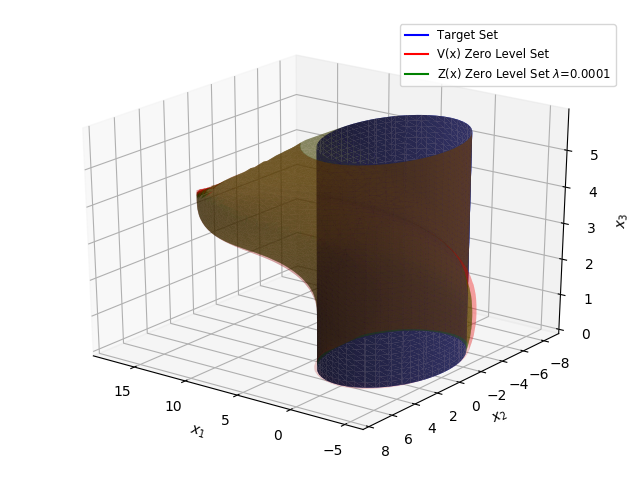}
        %\caption{}
    \end{subfigure}%
    \caption{The target set $\mathcal{T}$ (blue cylinder), zero sub-level sets of $V$ (red) and $Z$ (green) shown from three different perspectives. The discount rate for $Z$ is $\lambda=0.01$. Note that the zero sub-level set of $Z$ is a subset of the zero sub-level set of $V$.}
    \label{fig:air3D}
\end{figure*}

In the first experiment we compare a multigrid approach to value iteration.  The results are shown in Table~\ref{tab:multigrid_pe}. The experiment and table follows the same structure used for the double integrator model. Similar to the double integrator model, the multigrid approach outperforms value iteration for the pursuit-evasion game.

\begin{table}
\centering
\caption{Pursuit Evasion: Value iteration (VI) with Multigrid} CPU time in seconds
\begin{tabular}{|c| c| c| c| c| }
\hline
\# nodes & Coarse grid & Fine grid &  Fine grid-CS & Multigrid \\ \hline
$40^3$ & $2.068$ & $29.684$ & $24.320$ & $26.388$ \\ \hline
$80^3$ & $33.166$ & $352.099$ & $317.444$ & $350.610$\\ \hline
\end{tabular}
\label{tab:multigrid_pe}
\end{table}

We now construct two other models by tweaking $M_n$: setting $u_{max}=1.5$, which gives the evader an advantage, we get model $M_e$, and setting $d_{max}=1.5$, which gives the pursuer an advantage,  we get model $M_p$. In the final experiment we look at the impact of initializing value iteration with a solution from a similar model. Just like in the previous benchmark example, this experiment is motivated by Section \ref{sec:model_based}, where now $M_e$ and $M_p$ represent two possible models inferred from the system observations. In this context we have just ``learned" that the evader/pursuer is more maneuverable ($M_e$/$M_p$). We compute both value functions with and without setting the initialization to the solution for $M_n$. Again, we refer to this initialization as a warm start. The results are shown in Table~\ref{tab:ws_pe}.

\begin{table}
\centering
\caption{Pursuit Evasion: Value iteration (VI) with Warm Start (WS)} CPU time in seconds
\begin{tabular}{|c| c| c| c| c| c|}
\hline
\# nodes & $M_n$ & $M_e$ &  $M_e$-WS & $M_p$ & $M_p$-WS \\ \hline
$40^3$ & $35.043$ & $32.242$ & $21.483$ & $26.319$ & $23.366$ \\ \hline
$80^3$ & $439.751$ & $416.965$ & $308.821$ & $300.847$ & $296.568$\\ \hline
\end{tabular}
\label{tab:ws_pe}
\end{table}

%% file: conclusions.tex
% !TEX root = main_min_disc_dist.tex

We have presented a novel minimum discounted reward HJ formulation for approximating reachable sets. The main advantage of this new formulation over previous work is that the solution can be obtained as the unique fixed point to a contraction mapping. We also showed how other solutions like policy iteration, and multigrid approaches can be used to yield faster convergence. 

The benefits listed so far, are  within the context of the traditional control paradigm, where we have or assume a fixed model of the system under consideration. However, as learning and data-driven approaches become more powerful and pervasive, perhaps the greatest contribution of this work is that it can lie somewhere in between traditional control and model-free reinforcement learning. The approach is certainly model-based, but because of its agnosticism to initialization it also has the flexibility to incorporate data and build towards solutions iteratively. We foreshadowed how this can be done with temporal difference learning, and in the future we plan on exploring how RL algorithms can be used with this formulation to approximate reachable sets for systems with unknown models or that are high-dimensional. 

%% file: biographies.tex
% !TEX root = main_min_disc_dist.tex

\begin{IEEEbiography}[{\includegraphics[width=1in,height=1.25in,clip,keepaspectratio]{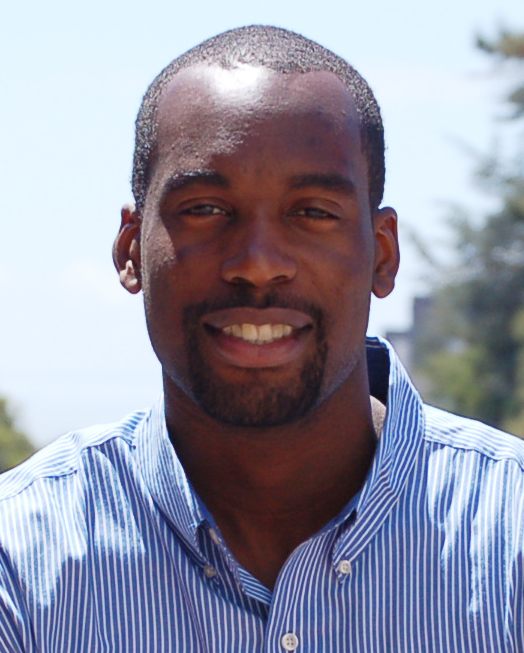}}]{Anayo K. Akametalu} is a PhD. candidate in Electrical Engineering and Computer Sciences at the University of California, Berkeley. He obtained his B.S. degree in Electrical Engineering from the University of California, Santa Barbara in 2012. His research interests lie at the intersection of control theory and reinforcement learning. He has been funded through the National Science Foundation Bridge to Doctorate Fellowship, UC Berkeley Chancellor's Fellowship, and GEM Fellowship.
\end{IEEEbiography}
\vspace{-3cm}
\begin{IEEEbiography}[{\includegraphics[width=1in,height=1.25in,clip]{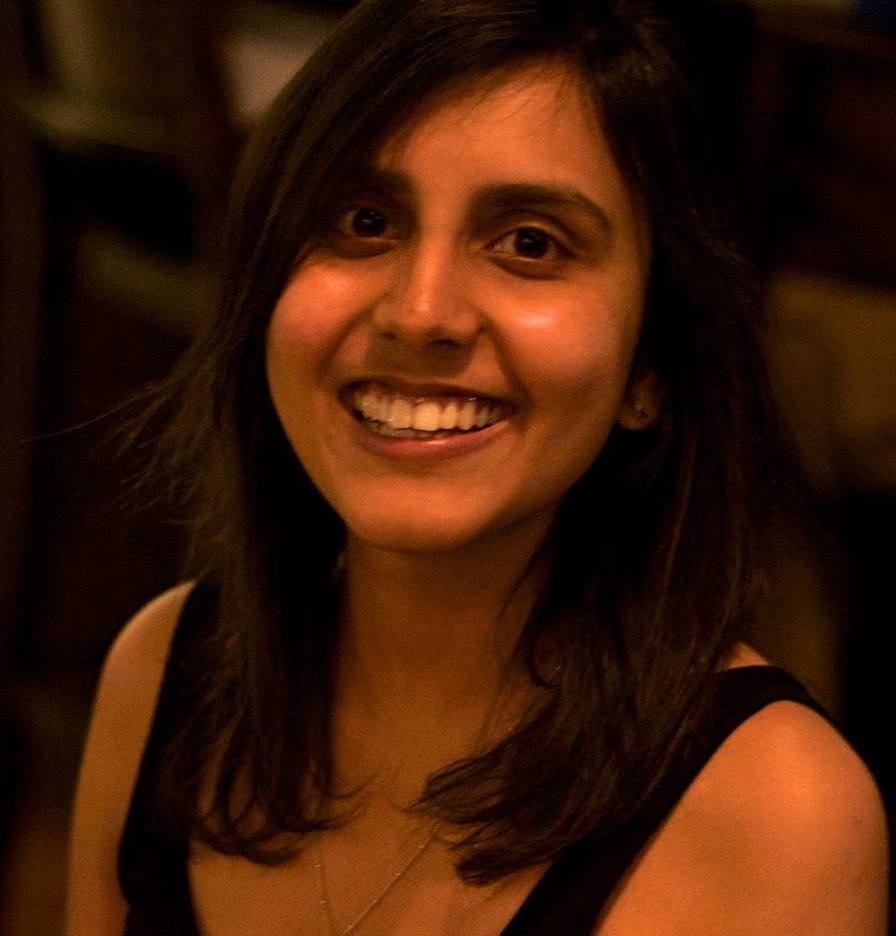}}]{Shromona Ghosh} received her Bachelor in Technology in Electronics and Communication Engineering from National Institute of Technology, Karnataka in 2013. She is currently a PhD candidate at University of California, Berkeley . Her research interests lie in the intersection of Formal Methods, Control Theory and Machine Learning. Specifically, she is looking into developing tools for the formal analysis of systems with learning components. 
\end{IEEEbiography}
\vspace{-3cm}
\begin{IEEEbiography}[{\includegraphics[width=1in,height=1.25in,clip,keepaspectratio]{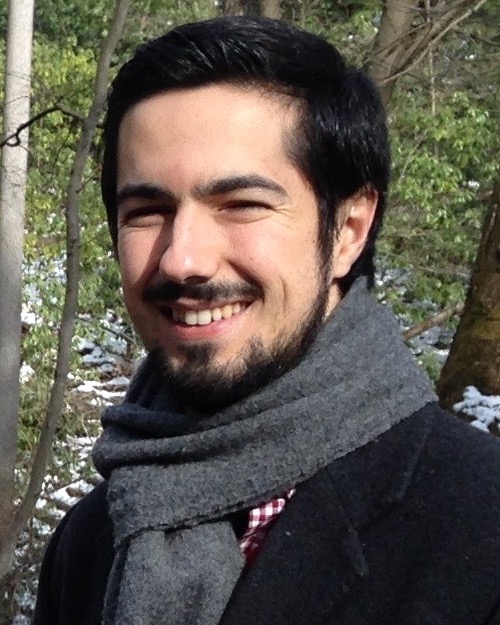}}]{Jaime F. Fisac} is a Ph.D. candidate in Electrical Engineering and Computer Sciences at the University of California, Berkeley. He received a B.S./M.S. degree in Electrical Engineering from the Universidad Polit{\'e}cnica de Madrid, Spain, in 2012, and a M.Sc. in Autonomous Vehicle Dynamics and Control from Cranfield University, UK, in 2013. He is a recipient of the ``la~Caixa'' Foundation Fellowship (2013-2015). His research interests lie in control theory, artificial intelligence, and cognitive science, with a focus on safety for robotic and AI systems operating closely with people.
\end{IEEEbiography}
\vspace{-3cm}
\begin{IEEEbiography}[{\includegraphics[width=1in,height=1.25in,clip,keepaspectratio]{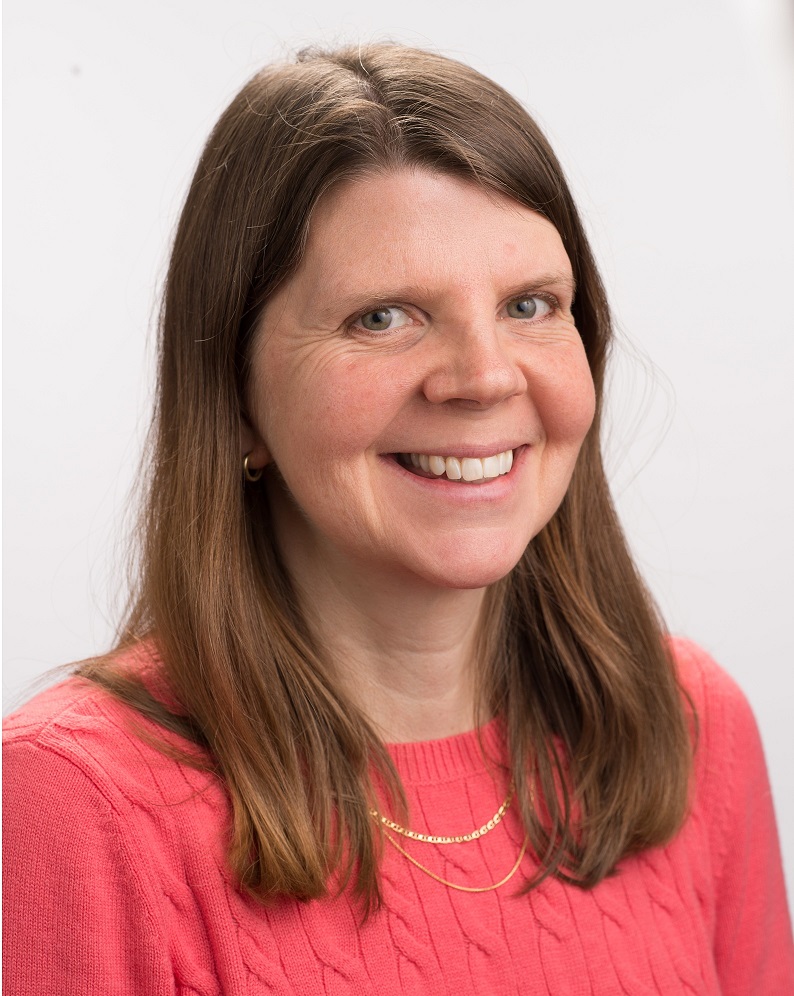}}]{Claire J. Tomlin} is the Charles A. Desoer Professor of Engineering in Electrical Engineering and Computer Sciences at the University of California, Berkeley. She was an Assistant, Associate, and Full Professor in Aeronautics and Astronautics at Stanford from 1998 to 2007, and in 2005 joined Berkeley. Claire works in the area of control theory and hybrid systems, with applications to air traffic management, UAV systems, energy, robotics, and systems biology. She is a MacArthur Foundation Fellow (2006) and an IEEE Fellow (2010), and in 2010 held the Tage Erlander Professorship of the Swedish Research Council at KTH in Stockholm.
\end{IEEEbiography}